\documentclass[12pt]{amsart}

\usepackage[utf8]{inputenc}
\usepackage[T1]{fontenc}

\usepackage{verbatim} % for comment environment
\usepackage[matrix, arrow, cmtip]{xy}
\usepackage{tikz}
\usetikzlibrary{matrix, positioning, arrows.meta}

\usepackage{amssymb, enumerate, enumitem, longtable, float, graphicx, array}

\makeatletter
\@namedef{subjclassname@2010}{%
  \textup{2010} Mathematics Subject Classification}
\makeatother

\newtheorem{theorem}{Theorem}[section]
\newtheorem{lemma}[theorem]{Lemma}

\newtheorem{conjecture}[theorem]{Conjecture}
\newtheorem*{theorem*}{Theorem}

\theoremstyle{definition}
\newtheorem*{definition}{Definition}

\theoremstyle{remark}

\newtheorem*{remark*}{Remark}

\newtheorem{construction}[theorem]{Construction}

\numberwithin{equation}{section}

\DeclareMathOperator{\Cl}{Cl}
\DeclareMathOperator{\Int}{Int}
\DeclareMathOperator{\diam}{diam}
\DeclareMathOperator{\im}{im}
\DeclareMathOperator{\dist}{dist}

\begin{document}

\bibliographystyle{abbrv}
\baselineskip=17pt

\title{Compactifications of unstable N\"obeling~spaces}

\author[A.~Nagórko]{A.~Nagórko}
\address[A.~Nagórko]{Faculty of Mathematics, Informatics, and Mechanics, University of Warsaw, Banacha 2, 02-097 Warszawa, Poland}
\email{amn@mimuw.edu.pl}
\thanks{This research was supported by the NCN (Narodowe Centrum Nauki) grant no. 2011/01/D/ST1/04144.}

\date{}

\begin{abstract}
We construct and embedding of a N\"obeling space $N^n_{n-2}$ of codimension $2$ into a Menger space $M^n_{n-2}$ of codimension $2$. This solves an open problem stated by R.~Engelking in 1978~\cite{engelking1978} in codimension~$2$.
\end{abstract}
\maketitle

\section{Introduction}

  A N\"obeling space $N^n_m$ is a subset of $\mathbb{R}^n$ consisting of points with at most~$m$ rational coordinates. If $n \geq 2m + 1$, then $N^n_m$ is homeomorphic to $\nu^n = N^{2n+1}_n$, the universal
$n$-dimensional N\"obeling space. The space~$\nu^n$ is considered to be an $n$-dimensional analogue
of the Hilbert space~$\ell^2$. In particular, it is characterized by the following theorem~\cite{ageev2007a, ageev2007b, ageev2007c, levin2009, nagorkophd, nagorko2013}, which is in direct
analogy to the characterization of the Hilbert space given by Toruńczyk~\cite{torunczyk1981}.
\begin{theorem}
An $n$-dimensional Polish space is a N\"obeling manifold if and only if it is an absolute neighborhood
  extensor in dimension~$n$ that is strongly universal in dimension~$n$.
\end{theorem}
For $n < 2m+1$ we call $N^n_m$ an \emph{unstable} N\"obeling space. A characterization of unstable N\"obeling spaces have been long sought~\cite{chigogidze1996}. The problem was stated again in a recent work of Gabai~\cite{gabai2009}, where it was conjectured that boundaries at infinity of certain mapping class groups are homeomorphic to unstable N\"obeling spaces.
% TODO: napisać o które MCG chodzi, by fragment był bardziej widoczny
The problem of characterizating unstable N\"obeling spaces is open.

In $1929$ Menger proposed an axiomatic characterization of covering dimension $\dim$~\cite{menger1929}.
In~\cite{hurewiczwallman1941} Hurewicz and Wallman described a problem whether these axioms characterize $\dim$ as the most important open problem in dimension theory.
Today, it is still not known whether $\dim$ satisfies these axioms.
By a difficult theorem of \v{S}tan'ko~  \cite{stanko1971}, the problem is now reduced to proving the following conjecture.  
\begin{conjecture}\label{compactification}
  Every $m$-dimensional subset of $\mathbb{R}^n$ has an $m$-dimensional compactification that can be embedded into $\mathbb{R}^n$.
\end{conjecture}

Since $\nu^n \subset \mathbb{R}^{2n+1}$ is an universal space for $n$-dimensional separable metric spaces, Conjecture~\ref{compactification} is true for spaces of dimension $m$ such that $2m + 1 \leq n$. The problem is hard in unstable dimensions, i.e. dimensions $m$ such that $2m + 1 > n$.
If $A$ is an $m$-dimensional subset of $\mathbb{R}^n$, then we say that $A$ is of \emph{codimension} $n - m$.
Conjecture~\ref{compactification} is trivial for codimension $0$. In codimension $1$, it is a theorem of Sierpiński~\cite{engelking1978}. The status of the conjecture is not known for codimension $2$. Since for $n \leq 3$ codimension $2$ is a stable case, the lowest dimensions where the conjecture is open are $m = 2$ and $n = 4$.

The main part of conjectured characterization of unstable N\"obeling spaces is the following conjecture~\cite{chigogidze1996}.
\begin{conjecture}
  Every $m$-dimensional subset of $\mathbb{R}^n$ can be embedded into $N^n_m$.
\end{conjecture}

It raises the following question, which was stated by Engelking in $1978$~\cite{engelking1978}: 
\begin{quote}
Does $N^n_m$ have $m$-dimensional compactification that can be embedded into $\mathbb{R}^n$?
\end{quote}
In this paper we answer this question in affirmative in codimension $2$. We prove the following theorem.

\begin{theorem}\label{thm:engelking}
  For each $n$ the N\"obeling space $N^n_{n-2}$ has an $(n-2)$-dimensional compactification that embeds into $\mathbb{R}^n$.
\end{theorem}

\section{Preliminaries}

In this section we review some conditions for a set $A \subset \mathbb{R}^n$ to be of codimension $\geq 2$. 
Recall a necessary and sufficient condition for a subset $X \subset \mathbb{R}^n$ to be of codimension $\geq 1$.

\begin{lemma}
  Let $X \subset \mathbb{R}^n$.
  We have $\dim X \leq n - 1$ if and only if for each open non-empty $U \subset \mathbb{R}^n$ the set $U \setminus X$ is non-empty.
\end{lemma}

There is an analoguous condition for compact subsets of codimension $\geq 2$.
This fails for non-compact subsets: recall the Sitnikov example~\cite{engelking1995} of a $2$-dimensional subset $X$ of $\mathbb{R}^3$ with the property that for each open and connected $U \subset \mathbb{R}^3$ the set $U \setminus X$ is non-empty and connected.

\begin{lemma}
  Let $X \subset \mathbb{R}^n$ be a compact subset.
  We have $\dim X \leq n - 2$ if and only if for each open non-empty connected $U \subset \mathbb{R}^n$ the set $U \setminus X$ is non-empty and connected.
\end{lemma}

Note that analoguous condition fails in codimension $3$: Antoine necklace is a $0$-dimensional compact subset of $\mathbb{R}^3$ with non-simply connected complement.

%To characterize sets of codimension $\geq 2$ in $\mathbb{R}^n$ we need stronger notion of connectedness.

%\begin{lemma}
%  Let $X \subset \mathbb{R}^n$.
%  We have $\dim X \leq n - 2$ if and only if complement of $X$ is dense and for each connected open subset $U$ of $\mathbb{R}^n$ and any two points $x, y \in U \setminus X$ there exists a continuum in $U \setminus X$ that contains $x$ and $y$.
%\end{lemma}

In the proof We will use the following position property of $N^n_k$.

\begin{definition}
  A space $X$ is $k$-connected if its homotopy groups of dimensions less than $k$ vanish.
\end{definition}

\begin{lemma}\label{lem:tame}
  Let $U$ be an open $k$-connected subset of $\mathbb{R}^n$.
  Let $m = n - k$ and let $N^n_m$ be a N\"obeling space of codimension $k$.
  Let $l < k$.
  \begin{enumerate}
  \item $\pi_l(U \setminus N^n_m) = 0$.
  \item If $2l + 1 < n$ and $\varphi \colon S^l \to U \setminus N^n_m$ is an embedding, then there exists
    a map $\varPhi \colon B^{l+1} \to U \setminus N^n_m$ such that $\varPhi_{|S^l} = \varphi$.
  \end{enumerate}
  In codimension $2$ we have the following additional property.
  \begin{enumerate}\addtocounter{enumi}{2}
  \item The map $\varPhi$ can be chosen to be an embedding.
  \end{enumerate}
\end{lemma}

\begin{definition}
  We say that a set $A \subset \mathbb{R}^n$ is \emph{$k$-tame} if it satisfies conditions (1) , (2) and (3) of lemma~\ref{lem:tame}.
\end{definition}

\section{A limit theorem}

We begin by stating a result that states sufficient conditions for a limit of a sequence of embeddings into $\mathbb{R}^n$ to be an embedding into $\mathbb{R}^n$. The key idea is to require that the inverses are uniformly continuous.

\subsection{The Game}
Let $X$ be a metric space and let $Y$ be a complete space.
Let $f_0 \colon X \to Y$ be an embedding such that $f_0^{-1}$ is uniformly continuous.
Let $Y_0 = f_0(X)$.
Consider the following infinite two-player game.
Player I starts and the players alternate moves under the following rules for move number $k \geq 1$.
\begin{enumerate}
\item Player I moves by selecting $\varepsilon_k > 0$.
\item Player II moves by selecting $g_k \colon Y_{k-1} \to Y$ such that 
\begin{enumerate}
  \item $g_k$ is an embedding that is $\varepsilon_k$-close to the identity
  \item $g_k^{-1}$ is uniformly continuous. 
\end{enumerate}
  We let $f_k = g_k \circ f_{k-1}$ and $Y_{k} = g_k(Y_{k-1}) = f_k(X)$.
\end{enumerate}
Player I wins if $\lim f_k \colon X \to Y$ exists and is an embedding of $X$ into $Y$.

\begin{theorem}\label{thm:limit}
  Player I has a winning strategy.
\end{theorem}
\begin{proof}
  Fix $k \geq 1$.
  By the rules of the game, $f_{k-1}$ is an embedding such that $f_{k-1}^{-1} \colon X_{k-1} \to X$ 
    is uniformly continuous.
  Let $\delta_{k} > 0$ such that if $y, y' \in X_{k-1}$ and $d_Y(y,y') < \delta_k$, then $d_X(f_{k-1}^{-1}(y),
    f_{k-1}^{-1}(y') < \frac 1{2^k}$.
  We let $\varepsilon_k = \frac{1}{2^k} \min \{ 1, \delta_k \}$.
  
  Since $Y$ is complete and $f_k$ is a Cauchy sequence, $f = \lim f_k$ exists and is continuous.
  
  Let $x, x' \in X$ and let $k$ such that $d_X(x,x') > \frac 1{2^k}$.
  Let $C = \prod_{m \geq 1} (1 - \frac{1}{2^m})$. We have $C > 0$. We will prove that
  \[
    d_Y(f(x), f(x')) \geq C \cdot \delta_k.
  \]
  This inequality implies both that $f$ is one-to-one and that $f^{-1}$ is (uniformly) continuous.
  
  By the choice of $\delta_k$, we have $d_Y(f_{k-1}(x), f_{k-1}(x')) \geq \delta_k$.
  Since $g_k$ is $\frac{\delta_k}{2^k}$-close to the identity, 
    \[ d_Y(f_k(x), f_k(x')) \geq d_Y(f_{k-1}(x), f_{k-1}(x')) 
      - \frac{2 \delta_k}{2^k} \geq (1 - \frac{1}{2^{k-1}}) d_Y(f_{k-1}(x), f_{k-1}(x')). \]
   Therefore 
   \[ d_Y(f(x), f(x')) \geq d_Y(f_{k-1}(x), f_{k-1}(x')) \cdot \prod_{m \geq k-1} (1 - \frac{1}{2^m}) 
     \geq \delta_k \cdot C. \]
\end{proof}

\subsection{The Moves} In this section we describe two constructions that we will use to play The Game as Player II. Since Player I wins the game, in the next section we'll use these moves to construct an embedding that proves the Main Theorem of the paper.

\subsubsection{\emph{Straightening} Move}

%\begin{tikzpicture}
%\draw plot [smooth, draw=white,double=red,very thick] coordinates {(0,0) (0.75, 1) (0.75, 0.75) (0.5,0.5) (0.4,0.75) (0.75,1.5) (1,2) };
%\draw [gray!50, xshift=4cm]  (0,0) -- (1,1) -- (2,-2) -- (3,0);
%\draw [cyan, xshift=4cm] plot [smooth, tension=2] coordinates { (0,0) (1,1) (2,-2) (3,0)};
%\end{tikzpicture}

\begin{construction}
   Let $A$ be an $1$-tame subset of $\mathbb{R}^n$ of finite diameter.
   Assume that $n > 3$.
   Let $L \subset \mathbb{R}^n$ $1$-dimensional hyperplane (a line).
   Let $\varepsilon > 0$.
   We construct a homeomorphism $\varphi \colon \mathbb{R}^n \to \mathbb{R}^n$ such that
   \begin{enumerate}
   \item $\varphi$ is the identity on a complement of a compact subset of $\mathbb{R}^n$.
   \item $\varphi$ is $\varepsilon$-close to the identity.
   \item $\im \varphi \cap L = \emptyset$.
   \end{enumerate}
\end{construction}
\begin{proof}
  Triangulate $L$ into segments and two half-lines in such a way that $0$-skeleton of $L$ misses $A$.
  Let $\tau(L)$ denote the triangulation.
  Because diameter of $A$ is finite, we may assume that both half lines are disjoint from the closure of $A$.
  Also we may assume that segments in $\tau(K)$ have lengths smaller than $1/2 \varepsilon$.
  For each segment $S \in \tau(L)$ let $C_S$ denote a cylindrical neighborhood of $\Int S$ in $\mathbb{R}^n$ such that $\diam C_S < \varepsilon$ and $S$ connects centers of opposite bases of cylinder $C_S$.
  Let $S'$ denote an embedded arc in $C_S$ that connects centers of opposite bases of $C_S$ and that is disjoint from $A$.
  It exists, because $A$ is $1$-tame.
  Since $n > 3$ we may unknot $S'$ in $C_S$ by a homeomorphism $h \colon C_S \to C_S$ that is the identity on $\partial C_S$ nad that maps $S'$ to onto $S$.
  We define $\varphi \colon \mathbb{R}^n \to \mathbb{R}^n$ by a formula $\varphi_{|C_S} = h_S$ on each $C_S$ and let it be the identity on the complement.
  Because diameter of each $C_S$ is smaller than $\varepsilon$, $\varphi$ is $\varepsilon$-close to the identity.
  Because support of $\varphi$ is compact, $\varphi^{-1}$ is uniformly continuous.
  By the assumption that $2m + 1 \leq n$ we can unknot the image of $\varphi$.
  We are done.
\end{proof}

\subsubsection{\emph{Push Away} Move}

\begin{construction}
  Let $L$ be an $m$-dimensional hyperplane in $\mathbb{R}^n$.
  Let $\varepsilon > 0$.
  We construct an embedding $\psi \colon \mathbb{R}^n \setminus L \to \mathbb{R}^n \setminus L$ such that
  \begin{enumerate}
  \item $\psi^{-1}$ is uniformly continuous.
  \item $\psi$ is $\varepsilon$-close to the identity.
  \item $\Cl (\im \psi) \cap L = \emptyset$.
  \end{enumerate}
\end{construction}
\begin{proof}
  Let $\mathbb{R}^n = L \times N$, where $N$ is $(n-m)$-dimensional hyperplane that is normal to $L$.
  Define $\psi \colon L \times (N \setminus \{ 0 \}) \to L \times N$ by the formula
  \[
    \psi(l, n) = (l, \xi(||n||) \cdot n),
  \]
  where $\xi \colon (0, \infty) \to (0, \infty)$ is given by the formula
  \[
    \xi(r) = \left\{
    \begin{array}{ll}
      \frac 12 (\varepsilon + r) & r \in (0, \varepsilon), \\
      r & r \in [ \varepsilon, \infty).
    \end{array}
    \right.
  \]
  Since $\xi^{-1}$ is $2$-Lipschitz, $\psi^{-1}$ is uniformly continuous.
\end{proof}

\section{Main Theorem}

Using the tools developed in previous sections we prove the Main Theorem of the paper.
\begin{theorem}
  \label{thm:embedding}
  If $A$ is a $2$-tame subset of $\mathbb{R}^n$ of codimension $2$ and $\diam A < \infty$, then the identity on $A$ can be arbitrarily closely approximated by embeddings into compact subsets of $N^n_{n-2}$.
\end{theorem}
\begin{proof}
  Let $l_1, l_2, \ldots$ be a sequence of $1$-dimensional hyperplanes such that $\bigcup_{i=1}^\infty l_i = \mathbb{R}^n \setminus N^n_{n-2}$. Fix $\varepsilon > 0$.
  
  We construct a sequence of embeddings $\varphi_i \colon A \to \mathbb{R}^n$ such that $\varphi_{i}$ is $\varepsilon_{i}$-close to $\varphi_{i-1}$ and $\varphi_{i-1} \circ \varphi_i^{-1}$ is uniformly continuous. The $\varepsilon_i$'s are selected according to the winning strategy of Player I from Theorem~\ref{thm:limit}. This guarantees that the limit $\varphi = \lim_{i \to \infty} \varphi_i$ exists and is an embedding of $A$ into $\mathbb{R}^n$. Additionally, we select $\varepsilon_i$'s so that $\sum_{i = 1}^\infty \varepsilon_i < \varepsilon$. Then $\varphi$ is $\varepsilon$-close to the identity on $A$.
  
  An embedding $\varphi_i$ is constructed so that
  \[
    \Cl (\im \varphi_i) \cap l_i = \emptyset.
  \]
  
  The construction is done in two steps. In the first step we use the Straightening Move to perturb $\varphi_{i-1}$ to an embedding $\varphi'_{i-i}$ such that $\im \varphi_i \cap l_i = \emptyset$. In the second step we compose $\varphi'_{i-1}$ with a Push Away move to construct the embedding $\varphi_i$ from $\varphi'_{i-1}$.

  Observe that $\im \varphi_i$ has finite diameter, so $\Cl \im \varphi_i$ is compact. Hence $\dist(\im \varphi_i, l_i) > 0$. Hence if subseqent $\varepsilon_i$'s are sufficiently close to each other, then the limit $\varphi$ satisfies $\dist(\im \varphi, l_i) > 0$. Then $\Cl \im \varphi$ is compact and disjoint from the complement of $N^n_{n-2}$ and therefore it is an embedding of $A$ into a compact subset of $N^n_{n-2}$.
\end{proof}

By lemma~\ref{lem:tame}, an unstable N\"obeling space $N^n_{n-2}$ is an $1$-tame $2$-codimensional subset of $\mathbb{R}^n$. It is homeomorphic to $N^n_2 \cap (0, 1)^n$, which is of finite diameter. By theorem~\ref{thm:embedding}, we can remebed it into a compact subset of $N^n_{n-2}$. The closure of the image of this embedding will be $(n-2)$-dimensional and compact. Therefore proof of Theorem~\ref{thm:engelking} is completed.

\bibliography{references}

\def\cprime{$'$}
\begin{thebibliography}{10}

\bibitem{ageev2007b}
S.~M. Ageev.
\newblock Axiomatic method of partitions in the theory of {N}\"obeling spaces.
  {II}. {A}n unknotting theorem.
\newblock {\em Mat. Sb.}, 198(5):3--32, 2007.

\bibitem{ageev2007c}
S.~M. Ageev.
\newblock Axiomatic method of partitions in the theory of {N}\"obeling spaces.
  {III}. {C}onsistency of the system of axioms.
\newblock {\em Mat. Sb.}, 198(7):3--30, 2007.

\bibitem{ageev2007a}
S.~M. Ageev.
\newblock The axiomatic partition method in the theory of {N}\"obeling spaces.
  {I}. {I}mproving partition connectivity.
\newblock {\em Mat. Sb.}, 198(3):3--50, 2007.

\bibitem{chigogidze1996}
A.~Chigogidze.
\newblock {\em Inverse spectra}, volume~53 of {\em North-Holland Mathematical
  Library}.
\newblock North-Holland Publishing Co., Amsterdam, 1996.

\bibitem{engelking1978}
R.~Engelking.
\newblock {\em Dimension theory}.
\newblock North-Holland Publishing Co., Amsterdam-Oxford-New York; PWN---Polish
  Scientific Publishers, Warsaw, 1978.
\newblock Translated from the Polish and revised by the author, North-Holland
  Mathematical Library, 19.

\bibitem{engelking1995}
R.~Engelking.
\newblock {\em Theory of dimensions finite and infinite}, volume~10 of {\em
  Sigma Series in Pure Mathematics}.
\newblock Heldermann Verlag, Lemgo, 1995.

\bibitem{gabai2009}
D.~Gabai.
\newblock Almost filling laminations and the connectivity of ending lamination
  space.
\newblock {\em Geom. Topol.}, 13(2):1017--1041, 2009.

\bibitem{hurewiczwallman1941}
W.~Hurewicz and H.~Wallman.
\newblock {\em Dimension {T}heory}.
\newblock Princeton Mathematical Series, v. 4. Princeton University Press,
  Princeton, N. J., 1941.

\bibitem{levin2009}
M.~Levin.
\newblock A {$Z$}-set unknotting theorem for {N}\"obeling spaces.
\newblock {\em Fund. Math.}, 202(1):1--41, 2009.

\bibitem{menger1929}
K.~Menger.
\newblock \"{U}ber die {D}imension von {P}unktmengen {III}.
\newblock {\em Monatsh. Math. Phys.}, 36(1):193--218, 1929.
\newblock Zur Begr{\"u}ndung einer axiomatischen Theorie der Dimension.

\bibitem{nagorkophd}
A.~Nag\'orko.
\newblock {\em Characterization and topological rigidity of {N}{\"o}beling
  manifolds}.
\newblock PhD thesis, University of Warsaw, 2006.

\bibitem{nagorko2013}
A.~Nag{\'o}rko.
\newblock Characterization and topological rigidity of {N}\"obeling manifolds.
\newblock {\em Mem. Amer. Math. Soc.}, 223(1048):viii+92, 2013.

\bibitem{stanko1971}
M.~A. {\v{S}}tan{\cprime}ko.
\newblock Solution of {M}enger's problem in the class of compacta.
\newblock {\em Dokl. Akad. Nauk SSSR}, 201:1299--1302, 1971.

\bibitem{torunczyk1981}
H.~Toru{\'n}czyk.
\newblock Characterizing {H}ilbert space topology.
\newblock {\em Fund. Math.}, 111(3):247--262, 1981.

\end{thebibliography}

\end{document}